\documentclass[english,11pt,reqno]{amsart}
\usepackage{amsmath}
\usepackage{latexsym}
\usepackage{amssymb}
\usepackage[colorlinks=true,linkcolor=black,citecolor=black]{hyperref}

\newcommand{\tad}{\, t_{a,d}}

\newcommand{\R}{\mathbb{R}}

\newcommand{\C}{\mathbb{C}}
\newcommand{\Q}{\mathbb{Q}}

\newcommand{\Z}{\mathbb{Z}}

\newcommand{\wT}{\widetilde{T}}

\newcommand{\tE}{\widetilde{E}}
\newcommand{\al}{\alpha}
\newcommand{\ga}{\gamma}
\newcommand{\de}{\delta}

\renewcommand{\b}{\mathfrak{b}}

\renewcommand{\le}{\leqslant}
\renewcommand{\ge}{\geqslant}

\newcommand{\sgn}{\mathrm{sgn }}
\newcommand{\adj}{\mathrm{Adj}}

\newcommand{\ev}{\mathrm{ev}}
\newcommand{\cI}{\mathcal{I}}
\newcommand{\cJ}{\mathcal{J}}

\newcommand{\TT}{\mathcal{T}}
\renewcommand{\SS}{\mathcal{S}}
\newcommand{\UU}{\mathcal{U}}

\newcommand{\VV}{\mathcal{V}}
\newcommand{\XX}{\mathcal{X}}
\newcommand{\NN}{\mathcal{N}}

\renewcommand{\adj}{\mathrm{adj}}

\newcommand{\comment}[1]{}

\theoremstyle{plain}
\newtheorem{theorem}{Theorem}
\newtheorem{corollary}[theorem]{Corollary}
\newtheorem{prop}[theorem]{Proposition}
\newtheorem{lemma}[theorem]{Lemma}

\numberwithin{equation}{section}
\theoremstyle{definition}

\title{An algebraic property of Hecke operators and two indefinite
theta series}
\author{Vicen\c{t}iu Pa\c{s}ol, Alexandru A. Popa}

\address{Institute of Mathematics ``Simion Stoilow" of the Romanian Academy,
P.O. Box 1-764, RO-014700 Bucharest, Romania}
\address{E-mail: vicentiu.pasol@imar.ro}
\address{E-mail: alexandru.popa@imar.ro}

\begin{document}

\begin{abstract}We prove an algebraic property of the elements defining Hecke operators on period polynomials
associated with modular forms, which implies that the pairing on period polynomials corresponding to the Petersson
scalar product of modular forms is Hecke equivariant. As a consequence of this proof, we derive two indefinite
theta series identities which can be seen as analogues of Jacobi's formula for the theta series associated with the
sum of four squares.
\end{abstract}
\maketitle

\section{Introduction}

The action of Hecke operators on period polynomials of cusp forms for the full modular group has been
defined algebraically by Choie and Zagier \cite{CZ, Z90,Za93}. To describe the elements acting as
Hecke operators on period polynomials, let $M_n$ be the set of integer matrices of determinant $n$, modulo $\pm I$,
and
let $R_n=\Q[M_n]$. Thus $\Gamma_1=\mathrm{PSL}_2(\Z)$ acts on $R_n$ by left and right multiplication. Let
$$M_n^\infty=\big\{\left(\begin{smallmatrix} a & b \\
0 & d \end{smallmatrix}\right):n=ad, 0\le b<d\big\}$$ be the usual system of representatives for
$\Gamma_1\backslash M_n$, and $T_n^\infty=\sum_{M\in M_n^\infty} M \in R_n$. Then there exist  $\wT_n, Y_n\in R_n$
such that
\begin{equation}\label{hecke}
T_n^{\infty}(1-S)=(1-S)\wT_n+(1-T)Y_n 
\end{equation}
where $S=\left(\begin{smallmatrix} 0 & -1 \\
1 & 0 \end{smallmatrix}\right)$ and $T=\left(\begin{smallmatrix} 1 & 1 \\
0 & 1 \end{smallmatrix}\right)$ are the standard generators of $\Gamma_1$.
Any such element $\wT_n$ gives the action of Hecke operators on period polynomials $P_f$ of cusp forms $f$ of
weight $k$ for the full modular group \cite{Z90}, namely  
$$P_{f|_k T_n} =P_f|_{2-k} \wT_n ,$$   
where $P_f(X)=\int_0^{i \infty} f(z)(z-X)^{k-2} dz$, and the Hecke operator $T_n$ acts on modular forms by $f|_k
T_n=n^{k-1} f|_{k} T_n^\infty$.  The slash operator  is defined for  $\gamma\in \mathrm{GL}_2(\R)$, 
and $w \in \Z$ by $f|_{w} \gamma (z)=f(\gamma
z) j(\gamma,z)^{-w}$,  with $j(\gamma,z)=cz+d$ if $\gamma$ has second row $(c,d)$.
The action is extended to elements of $R_n$ by linearity. 

It was shown later by Diamantis that the elements $\wT_n$ also give an
action of Hecke operators on (multiple) period polynomials attached to cusp forms for the congruence subgroup
$\Gamma_0(N)$, when $n$ and $N$ are coprime \cite{Di01}. In a recent paper \cite{PP}, we show that the same
elements have actions on period polynomials of arbitrary modular forms for finite index subgroups $\Gamma$ of
$\Gamma_1$, corresponding to actions of a wide class of double coset operators $\Sigma_n\subset M_n$ on modular
forms. Remarkably the elements $\wT_n$ are universal, not depending on $\Gamma$ or the double coset
$\Sigma_n$. 

Elements $\wT_n$ satisfying condition \eqref{hecke} go back to work of Manin \cite{M73}. The element $\wT_n$ is
unique, up to addition of any element in the right $\Gamma_1$-module 
\begin{equation*}
 \cI=(1+S)R_n+(1+U+U^2)R_n,
\end{equation*}
where $U=TS$ has order 3. Examples of such $\wT_n$ are given in \cite{CZ,Z90}. Another
example will be given in an upcoming article of the second author and Don Zagier, leading to a simple proof
of the Eichler-Selberg trace formula for modular forms on $\Gamma_1(N)$ with Nebentypus, which generalizes the
approach sketched in \cite{Za93} for $\Gamma_1$.            

In the context of modular symbols, Merel gives many examples of elements
$\widetilde{\mathfrak{T}}_n\in R_n$ expressing the action of Hecke operators on modular symbols, which satisfy a
condition similar to \eqref{hecke} \cite{Me}. It can be shown that their adjoints
$\wT_n=\widetilde{\mathfrak{T}}_n^\vee$ satisfy \eqref{hecke}, where for $g\in M_n$ we denote by $g^\vee=g^{-1}\det
g $ the adjoint of $g$, and we apply this notation to all elements of $R_n$ by linearity.

The space of period polynomials is endowed with a pairing corresponding to the Petersson scalar product of modular
forms via a generalization of a formula of Haberland. We show in \cite{PP} that the Hecke equivariance of this
pairing under the action of $\wT_n$ is implied by the following algebraic property of the elements $\wT_n$. 
\begin{theorem}\label{thm_hecke}
For any element $\wT_n\in R_n$ satisfying property \eqref{hecke} we have
\begin{equation}\label{eq_hecke}
 \wT_n(T-T^{-1})+(T^{-1}-T)\wT_n^\vee \in \cI+\cI^\vee.
\end{equation}
\end{theorem}

The rest of the paper is devoted to an algebraic proof of the theorem, along the way obtaining other
properties of the elements defining Hecke operators, which may be of independent interest. An unexpected outcome of
this algebraic approach is the discovery of the following ``indefinite theta series'' identities, which are proved
in Section \ref{ssec_main}:
\begin{equation*}
\sideset{}{'}\sum_{\substack{a,b,c,d\ge 0\\ a+b>|d-c|,\ c+d>|a-b|}}\hspace{-8pt} q^{ad+bc}=
3 \tE_2(q) -2\sum_{n>0} \sigma_{\min}(n)\, q^n+\sum_{n>0} q^{n^2},
\end{equation*}
\begin{equation*}
\sideset{}{'}\sum_{\substack{x\ge|y|, z\ge
|t|\\x>|t|,z>|y|}}\hspace{-8pt}q^{x^2+z^2-y^2-t^2}=\tE_2(q)-2\tE_2(q^2)+4\tE_2(q^4)-2\sum_{n>0}
\sigma_{\min}^\ev(n)\, q^n+\sum_{n>0} q^{n^2},
\end{equation*}
where $\sigma_{\min} (n)=\sum_{n=ad} \min(a,d)$, $\sigma_{\min}^\ev (n)=\sum_{n=ad, 2|(d-a)} \min(a,d)$ (the sums
are over positive divisors),  
 $\tE_2(q)=\sum_{n>0} \sigma_1(n) q^n$ is the weight 2 Eisenstein series without the
constant term, and the notation $\sum'$ indicates that the terms for which there is equality in the
range are summed with weight 1/2. For example, when $p$ is an odd prime, we obtain
that $p-3$ is the number of integer solutions of 
$$x^2+z^2-y^2-t^2=p, \text{ with } x,z>|y|,|t|.$$  
This can be seen as an indefinite analogue of Jacobi's result on the number of ways of
representing a positive integer as a sum of four squares. It would be interesting to fit these
examples into a theory of theta series for indefinite quadratic forms, of the type developed by
Zwegers \cite{Zw} for forms of signature $(r,1)$.

The proof of Theorem \ref{thm_hecke} is organized as follows. In \S\ref{ssec_prelim} we derive a characterization
of $\cI+\cI^\vee$ that
allows us to reduce \eqref{eq_hecke} to a relation involving only $T_n^\infty$. This relation is then shown to be
equivalent to two relations possessing an extra symmetry besides invariance under taking adjoint, which are
proved in \S\ref{ssec_main}. The main part of the proof is contained in \S\ref{ssec_main}, where we also derive
the indefinite theta series identities as immediate corollaries. 

\section{Preliminary reductions} \label{ssec_prelim}
First we prove a characterization of the set $\cI+\cI^\vee$, based on a similar
characterization of $\cI$ in \cite{CZ}. 
\begin{prop}\label{lem1}
Let $A\in R_n=\Q[M_n]$. Then $A\in \cI+\cI^\vee$ if and only if 
$$(1-S)  A  (1-S)\in (1-T)  R_n   (1-S)+ (1-S)  R_n  
(1-T^{-1}).$$ 
\end{prop}  
\begin{proof}
The proof is based on the characterization of $\cI$ in \cite[Lemma 3]{CZ}: 
\begin{equation}\label{lemma3}
 v\in \cI \Longleftrightarrow (1-S)v\in (1-T) R_n.
\end{equation}
We also need the following statement, which appears in the proof of Lemma 3 in
\cite{CZ}:
\begin{lemma}[\cite{CZ}]\label{lem_cz} Let $A,B\in R_n$ such that $(1-S) A= (1-T) B$. Then
there exists $C\in R_n$ such that 
\[A=(1+S)C-S B, \text{ with } SB\in (1+U+U^2)R_n.  
\] 
\end{lemma}
If $A\in \cI+\cI^\vee$, the claim of Proposition \ref{lem1} follows immediately from
\eqref{lemma3} and its adjoint. Assume therefore that $A\in R_n$ satisfies:
\[(1-S)  A  (1-S)=(1-T)  \alpha  (1-S)+ (1-S)  \beta  (1-T^{-1})
\]
By the adjoint of the relations in Lemma \ref{lem_cz} it follows that there exists $C\in
R_n$ such that 
\begin{equation}\label{eq1}
 (1-S)  A=(1-T)  \al+C(1+S)+(1-S)  \beta S, 
\end{equation} 
with $(1-S)\beta S\in R_n(1+U+U^2)$.  
Since $(1\pm S)^2=2 (1\pm S)$, multiplying \eqref{eq1} by 2 we get
$$
(1-S)\cdot 2A=(1-T)\cdot  2\al+C(1+S)  (1+S)+(1-S) (1-S)\beta S
$$
From \eqref{eq1} we can write $C(1+S)=(1-S)  \ga+ (1-T)  \de$, therefore:
$$
(1-S) [2A-\ga  (1+S)-(1-S)\beta S]\in (1-T)R_n.
$$
By \eqref{lemma3} we conclude
$$
2A-\ga  (1+S)-(1-S)\beta S \in \cI.
$$
Since $(1-S)\beta S\in R_n(1+U+U^2)$, it follows by dividing the last equation by 2 that
$A\in \cI+\cI^\vee$. 
\end{proof}
We recall from \cite[Thm. 2]{CZ} that, in addition to the
relation \eqref{hecke}, the element $T_n^\infty\in R_n$ also satisfies 
\begin{equation} \label{2.1}
 T_n^\infty (1-T) \in (1-T) R_n.
\end{equation}
\begin{corollary}\label{cor5.3}
Let $\cJ=(1-T)R_n(1-S)\subset R_n$. Then \eqref{eq_hecke} is equivalent to the statement
\begin{equation}\label{h1}
 T_n^\infty ST^{-1}(1-S)  +(1-S)TS T_n^{\infty \vee}\in \cJ+\cJ^\vee.
\end{equation}
\end{corollary}
\noindent We remark that $T_n^\infty A (1-S) \pmod{\cJ}$ is
well-defined modulo multiplication by powers of $T$ on the left, for any $A\in\Gamma_1$, so we
can take in \eqref{h1} \begin{equation}\label{7.3}
T_n^\infty=\sum_{n=ad}\sum_{b\!\!\!
\pmod{d}} t_{a,d}(b), \quad \text{ where }  
t_{a,d}(b)=\begin{pmatrix} a& b \\0 & d \end{pmatrix}.
\end{equation}
\begin{proof}
By Proposition \ref{lem1} and \eqref{hecke}, 
\eqref{eq_hecke} is equivalent to the following statement:
 \begin{equation*} 
    T_n^\infty(1-S)(T-T^{-1})(1-S)+(1-S)(T^{-1}-T)(1-S){T_n^\infty}^\vee \in \cJ+\cJ^\vee.
\end{equation*}
By \eqref{2.1}, we have $T_n^\infty(1-T)(1-S)\in \cJ$, so that: 
\begin{equation*}
\begin{split}
T_n^\infty (1-S)(T-T^{-1})(1-S) & \equiv T_n^\infty S(T^{-1}-T)(1-S)\\
 & \equiv T_n^\infty (ST^{-1}+T^{-1}ST^{-1})(1-S)\\
 & \equiv 2 T_n^\infty ST^{-1}(1-S)\quad \pmod{\cJ}
\end{split} 
\end{equation*}
On the second line we used $STS=T^{-1}ST^{-1}$, and on the third the remark above.
\end{proof}

We need the following notation. Let
$\epsilon=\left(\begin{smallmatrix} -1 & 0 \\ 0 & 1 \end{smallmatrix}\right) $ and for $m\in
M_n$ denote $m'=\epsilon m\epsilon$. For any matrix $m\in M_n$ we
denote by $m^\star\in R_n$ the following linear combination of eight matrices: 
\begin{equation}\label{star}
m^{\star}=(m-m')(1-S)+(1-S)(m^\vee-m^{\prime\vee} )
\end{equation}
and extend this notation to elements of $R_n$ by linearity. Note that
$m^{\prime\vee}=m^{\vee\prime}$. 

In the next proposition we show that \eqref{h1} is equivalent with more symmetric relations, 
involving only terms of type $m^\star$. The latter will be proved in the next section,
using only the following easily checked properties:
\begin{gather}
 m^\star+(m^{\prime\vee})^\star+(SmS)^\star+(Sm^{\prime\vee}S)^\star=0\tag{P1}\label{P1}\\
 (m')^\star=-m^\star\tag{P2}\label{P2}\\
 (mS)^\star=-m^\star\tag{P3}\label{P3}\\
 (T^k m)^\star \equiv m^\star \pmod{\cJ+\cJ^\vee} \quad \text{for any integer
$k$.}\tag{P4}\label{P4}
\end{gather}

\begin{prop} \label{prop3} Let $U_n=T_n^\infty U^2(1-S)+(1-S)U T_n^{\infty
\vee}\in R_n$.  We have the following congruences $(\mathrm{mod}\ \cJ+\cJ^\vee)$:
\begin{gather} \label{5.7}
\sum_{n=ad}\Big[ \sum_{-d-\frac{a}{2} <k<d-\frac{a}{2}}
\left(\begin{smallmatrix} a+k & -k\\-d & d
\end{smallmatrix}\right)-\delta\big(\textstyle\frac{a}{2}\big)
\big(\begin{smallmatrix} d & -d\\ a/2 & a/2 \end{smallmatrix}\big)
\Big]^\star \equiv 4 U_n \\
\sum_{n=ad} \Big[2 \sum_{-\frac{a}{2}< k<\frac{d-a}{2}} \left(\begin{smallmatrix} d & -d\\-k
& a+k
\end{smallmatrix}\right)+\delta\big(\textstyle\frac{a}{2}\big)
\big(\begin{smallmatrix} d & -d\\ a/2 & a/2 \end{smallmatrix}\big)\Big]^\star \equiv 2
U_n, \label{5.81}
\end{gather}
where $\delta(x)$ is 1 if $x\in\Z$ and 0 otherwise.
\end{prop}
\noindent The LHS of \eqref{5.7} is congruent to $2 (T_n^\infty U^2)^\star$, while the LHS
of \eqref{5.81} is congruent with $(U T_n^\infty U^2)^\star$, but we will not need these facts
in the sequel. 
\begin{proof}
The summand in the LHS of \eqref{5.7} is $(t_{a,d}(-k-a)ST^{-1})^\star$ [with $\tad(b)$ defined
in
\eqref{7.3}], and we note that $k$ is
well-defined modulo $d$ by \eqref{P4}. We can rewrite the LHS as follows:
\[
\sum_{n=ad}\sum_{k\ \textrm{mod}\ d}
2(t_{a,d}(k)ST^{-1})^\star-\delta\big(\textstyle\frac{a}{2}\big)
\big[\big(\begin{smallmatrix} a/2 & a/2\\ -d & d \end{smallmatrix}\big)+
\big(\begin{smallmatrix} d & -d\\ a/2 & a/2 \end{smallmatrix}\big) \big]^\star\equiv 2
(T_n^\infty ST^{-1})^\star,
\]
because the sum of the terms in square brackets equals 
\[\sum_{n=2ad} \big[\big(\begin{smallmatrix} a & a\\ -d & d \end{smallmatrix}\big)+
\big(\begin{smallmatrix} d & -d\\ a & a \end{smallmatrix}\big) \big]^\star\equiv 0
\]
(exchanging $a$ with $d$ in the first matrix and using \eqref{P2}). 

For any $A\in R_n$, we use the abbreviation $\{A\}+\{\adj\}:= A+A^\vee$. We have
\begin{equation}\label{7.10}
\begin{split}
(T_n^\infty ST^{-1})^\star&\equiv \big\{[T_n^\infty
ST^{-1}-(T_n^{\infty})'ST](1-S)\big\}+\{\adj \} \\
&\equiv \big\{[T_n^\infty
ST^{-1}+(T_n^{\infty})STS](1-S)\big\}+\{\adj \} \\
&\equiv \big\{ 2 T_n^\infty ST^{-1}(1-S)\big\}+\{\adj \} ,
\end{split}
\end{equation}
proving \eqref{5.7}. On the second line we have used the remark following
Corollary \ref{cor5.3}, while on the third we used $(ST)^3=1$ together with \eqref{2.1}.  

The summand in \eqref{5.81} is $(S \tad(k) ST^{-1})^\star$, and we have as in \eqref{7.10}:
\[
(S \tad(k) ST^{-1})^\star=\big\{(S \tad(k) ST^{-1}+S\tad(-k-a)ST^{-1})(1-S)\big\}+\{\adj\}.
\]
The range $-\frac{a+d}{2}< k<\frac{d-a}{2}$ is invariant under $k\rightarrow -k-a$, hence the
LHS of \eqref{5.81}, which we denote $L_n$, becomes:

\begin{equation}\label{5.13}
L_n\equiv \sum_{n=ad}  \Big\{ 2
\sum_{-\frac{a+d}{2}<k<\frac{d-a}{2}} S \tad (k)ST^{-1}(1-S)\Big\}+\{\adj\}.
\end{equation}
Now we have:
\begin{equation}\label{5.14}
S\tad(k)ST^{-1}(1-S)\equiv \tad(k)ST^{-1}(1-S)+(1-S)[\tad(a+k)-\tad(k)]S . 
\end{equation}
Since $\tad(k)ST^{-1}(1-S)\pmod{\cJ}$ depends only on $k\!\! \pmod{d}$ we have:
\begin{equation}\label{5.15}
 \sum_{-\frac{a+d}{2}<k}^{k<\frac{d-a}{2}} \tad (k)U^2(1-S)\equiv
\sum_{k=0}^{d-1}\tad(k) U^2(1-S)-\delta\big(\textstyle\frac{d-a}{2}\big)\tad
\big(\frac{d-a}{2}\big)U^2(1-S).
\end{equation}

Similarly 
\[
\sum_{-\frac{a+d}{2}<k}^{k<\frac{d-a}{2}} [t_{a,d}(a+k)-t_{a,d}(k)] (S-1) \equiv
\delta\big(\textstyle\frac{d-a}{2}\big)\big[\tad \big(\frac{d-a}{2}\big) -
\tad \big(\frac{a-d}{2}\big)
\big] (S-1)
\]

Next we observe that
\begin{equation}
\begin{split}
 &\sum_{n=ad} \sum_{-\frac{a+d}{2}<k}^{k<\frac{d-a}{2}}
\{\tad(k)-\tad(a+k)\}+\{\adj\}\equiv \\ &\equiv  \sum_{n=ad}\Big[ \sum_{-\frac{a+d}{2}<k}^{k<\frac{d-a}{2}}
[\tad(k)-\tad(a+k)] +\sum_{-\frac{a+d}{2}<k}^{k<\frac{a-d}{2}}
[\tad(-k)-\tad(-d-k)]\Big] \\ &\equiv 
\sum_{n=ad} 
\delta\big(\textstyle\frac{d-a}{2}\big)
\big[\tad \big(\frac{a-d}{2}\big) -
\tad \big(\frac{d-a}{2}\big)\big].\label{5.16}
\end{split}
\end{equation}
Notice that \eqref{5.16} can be conjugated by $S$, since $S^\vee=S$. Putting together the
resulting equation, and \eqref{5.16}, \eqref{5.15}, \eqref{5.14}, \eqref{5.13}, we obtain:
\[L_n \equiv 
\{2T_n^\infty ST^{-1}(1-S)\}+\{\adj\}+2\sum_{n=ad}
\delta\big(\textstyle\frac{d-a}{2}\big) E_{a,d}
\]
where:
\begin{multline*}
E_{a,d}=S\big[\tad \big(\textstyle\frac{a-d}{2}\big) - \tad \big(\frac{d-a}{2}\big)\big]S+
\tad \big(\frac{d-a}{2}\big) - \tad \big(\frac{a-d}{2}\big)+\\
+\big\{\big[\tad \big(\textstyle\frac{d-a}{2}\big) - \tad \big(\frac{a-d}{2}\big)\big] (S-1)-
\tad \big(\frac{d-a}{2}\big)U^2(1-S) \big\}+ \{\adj\}.
\end{multline*}
It is easy to see that $E_{a,d}+E_{d,a}\equiv 0 \pmod{\cJ+\cJ^\vee}$, finishing the proof of
\eqref{5.81}.
\end{proof}

\section{Main algebraic result and two theta series identities}\label{ssec_main}
In this section, we finish the proof of \eqref{eq_hecke}. The indefinite theta series
identities from the introduction follow as a bonus from the proof, and since they are of
independent interest we keep this section self-contained. 

Recall the notation \eqref{star} and properties \eqref{P1}-\eqref{P4}, which are the only
ingredients needed.  For convenience we rewrite the four-term relation:
\begin{equation}\tag{P1}\label{P11}
 \left(\begin{smallmatrix} a & -b\\c & d \end{smallmatrix}\right)^\star+
 \left(\begin{smallmatrix} d & -b\\c & a \end{smallmatrix}\right)^\star+
 \left(\begin{smallmatrix} a & -c\\b & d \end{smallmatrix}\right)^\star +
\left(\begin{smallmatrix} d & -c\\b & a \end{smallmatrix}\right)^\star =0.
\end{equation}
We will apply this relation to matrices in the following set
\[
\SS_n=\Big\{\left(\begin{smallmatrix} a & -b\\c & d \end{smallmatrix}\right)\in M_n : c\ge
b;\quad d\ge a;\quad
a+b>d-c> 0 \Big\}.
\]
For $\gamma= \left(\begin{smallmatrix} a & -b\\c & d
\end{smallmatrix}\right)$, denote $A_\gamma=\big\{\left(\begin{smallmatrix} a & -b\\c & d
\end{smallmatrix}\right),
 \left(\begin{smallmatrix} d & -b\\c & a \end{smallmatrix}\right),
 \left(\begin{smallmatrix} a & -c\\b & d \end{smallmatrix}\right) ,
\left(\begin{smallmatrix} d & -c\\b & a \end{smallmatrix}\right) \big\}$ and define 
\begin{equation}\label{5.17}
 \TT_n:=\displaystyle\bigcup_{\gamma\in\SS_n} A_\gamma =\Big\{\left(\begin{smallmatrix} a &
-b\\c & d \end{smallmatrix}\right):\ b+c>|a-d|,\quad
\max(a,d)> \max(b,c)\Big\}
\end{equation}
(the union is disjoint), so that we have by construction
\begin{equation}\label{5.18}
 \sum_{\gamma\in\TT_n}\gamma^\star=0. 
\end{equation}
It is convenient to further symmetrize the set $\TT_n$ by interchanging $a\leftrightarrow b$,
$c\leftrightarrow d$. Note that if $a,b,c,d$
satisfy $\max(a,d)\ge \max(b,c)$ then 
\[
b+c>|a-d| \Longleftrightarrow c+d>|a-b|, \quad a+b>|c-d|. 
\]
Defining therefore 
\begin{equation}
\begin{split}
 \UU_n&:=\TT_n\cup \Big\{ \left(\begin{smallmatrix} b & -a\\d & c \end{smallmatrix}\right):
\left(\begin{smallmatrix} a & -b\\c & d \end{smallmatrix}\right) \in\TT_n\Big\}\\
\XX_n&:=\Big\{\left(\begin{smallmatrix} a & -b\\c & d \end{smallmatrix}\right):\ c+d>|a-b|,
\quad a+b>|c-d|\Big\}\\
\VV_n&:=\Big\{\left(\begin{smallmatrix} a & -b\\c & d \end{smallmatrix}\right)\in \XX_n:
\max(a,d)=\max(b, c)\Big\},
\end{split}
\end{equation}
we have the disjoint union $\XX_n= \UU_n\cup\VV_n$.  We
have $\sum_{\gamma\in\UU_n}\gamma^\star=0$, by \eqref{P2}, \eqref{P3}, \eqref{5.18}, and it is easily found that 
\begin{equation}\label{5.19}
\VV_n=\bigcup_{\gamma} A_\gamma \quad \text{ over }
\gamma=\left(\begin{smallmatrix} a & -b\\c & c
\end{smallmatrix}\right) \text{ with } c|n, \ c\ge a \ge \frac{n}{c}-c
\end{equation}
so that  $\sum_{\gamma\in\VV_n}\gamma^\star=0$. Consequently
\begin{equation}\label{5.20}
 \sum_{\gamma\in\XX_n}\gamma^\star=0. 
\end{equation}
\begin{prop} \label{prop_alg} For every $n>0$ we have the congruence
$(\mathrm{mod}\ \cJ+\cJ^\vee)$:
 \begin{equation}\label{5.21}
\sum_{\gamma\in\XX_n}\gamma^\star\equiv\sum_{n=ad}\Big[\sum_{-d-\frac{a}{2}<k<d-\frac{a}{2}}
\left(\begin{smallmatrix} a+k & k\\d & d \end{smallmatrix}\right)+
2\sum_{-\frac{a}{2} <k<\frac{d-a}{2}} \left(\begin{smallmatrix} d & d\\k & a+k
\end{smallmatrix}\right)
\Big]^\star.
\end{equation}
\end{prop}
\noindent  By Proposition $\ref{prop3}$ and \eqref{P2}, we obtain  
$$0=\sum_{\gamma\in\XX_n}\gamma^\star\equiv-6\left[T_n^\infty U^2(1-S)+(1-S)U T_n^{\infty
\vee}\right] \pmod{\cJ+\cJ^\vee},$$ 
which, together with Corollary \ref{cor5.3}, finishes the proof of Theorem \ref{thm_hecke}. 

Before proving the proposition, we show that it implies the first theta series identity
in the introduction. For a set of 2 by 2 matrices $A$ denote by
$\NN(A)$ the number of matrices
$\left(\begin{smallmatrix} * & *\\c & d \end{smallmatrix}\right)\in A$ with $cd>0$, minus
the number of matrices with $cd<0$, namely:
\[
\NN(A)=\sum_{\gamma\in A} \sgn(c_\gamma d_\gamma).
\] 
\begin{corollary}\label{cor5.6} We have the identity:
\begin{equation}\label{5.22}
\sideset{}{'}\sum_{\substack{a,b,c,d\ge 0\\ a+b>|d-c|,\ c+d>|a-b|}}q^{ad+bc}=
3 \tE_2(q) -2\sum_{n>0} \sigma_{\min}(n)\, q^n+\sum_{n>0} q^{n^2},
\end{equation}
where $\sigma_{\min} (n)=\sum_{n=ad}
\min(a,d)$.
\end{corollary}
\begin{proof}
We claim that $\NN(\XX_n)$ equals the coefficient of $q^n$ in the LHS of \eqref{5.22}; on the
other hand, going from \eqref{5.20} to \eqref{5.21} we use only relations
\eqref{P2}-\eqref{P4}, which do not change the function $\NN$, so $\NN(\XX_n)$ can be computed
exactly from the RHS of \eqref{5.21}. Indeed counting matrices in the RHS of \eqref{5.21} yields 
\[
\NN(\XX_n)=[2\sigma(n)-\tau_{\ev}(n)]+[\sigma(n)+\tau_\ev(n)-2\sigma_{\min}(n)+\delta(\sqrt{n})]
\]
where the first term counts matrices in the first sum, and the second counts matrices in the second sum
(with sign depending on the sign of $k$), and
$\tau_\ev(n)$ denotes the number of even divisors of $n$. The result agrees with the coefficient of $q^n$ in the
RHS of \eqref{5.22}, and the proof is finished once we prove the claim above.

To count $\NN(\XX_n)$, we start with $\TT_n$. Let $\gamma=\left(\begin{smallmatrix} a & -b\\c &
d \end{smallmatrix}\right)\in \SS_n$. Clearly $c,d>0$ and we have three cases:
\begin{enumerate}
 \item $a>0>b$ or $b>0>a$. Then $\NN(A_\gamma)=0$; 
 \item $a,b>0$. Then $\NN(A_\gamma)=|A_\gamma|$;
 \item One of $a$, $b$ is 0. Then the other is positive and
$\NN(A_\gamma)=\frac{1}{2}|A_\gamma|$.
\end{enumerate}
By \eqref{5.17} we conclude that $\NN(\TT_n)=\#\big\{\left(\begin{smallmatrix} a & -b\\c & d
\end{smallmatrix}\right)\in \TT_n: a,b,c,d\ge 0\big\}$, each matrix with $abcd=0$ being
counted with weight 1/2. The same is obviously true for $\UU_n$, and by inspection for
$\VV_n$ (see \eqref{5.19}). We conclude that 
\begin{equation}\label{5.222}
\NN(\XX_n)=\#\Big\{\left(\begin{smallmatrix} a & -b\\c & d
\end{smallmatrix}\right)\in \XX_n: a,b,c,d\ge 0; \begin{array}{c}
\text{ matrices with $abcd=0$ are }\\
\text{ counted with weight 1/2 } \end{array}
 \Big\},
\end{equation}
that is $\NN(\XX_n)$ is the coefficient of $q^n$ in the LHS of \eqref{5.22}, as claimed.  
\end{proof}

\begin{proof}[Proof of Proposition \ref{prop_alg}.]
We decompose $\XX_n$ as a disjoint union:
\[
 \XX_n=\XX_n^<\cup\XX_n^=\cup\XX_n^>
\]
where $\XX_n^<$, $\XX_n^=$, and $\XX_n^>$ consist of those
$\gamma=\left(\begin{smallmatrix} a & -b\\c & d
\end{smallmatrix}\right)\in\XX_n$ with $|c|<d$, $c=d$, and $c>|d|$ respectively. Using
\eqref{P2}-\eqref{P4}, we will show that $\sum_{\gamma \in \XX_n} \gamma^\star$ reduces to
\eqref{5.21}. 

If $\left(\begin{smallmatrix} x & k \\d & d \end{smallmatrix}\right)\in \XX_n^=$, we have
$n=ad$, $x=a+k$ and $2d>|a+2k|$, hence
\begin{equation}\label{5.221}
 \sum_{\gamma\in\XX_n^=}\gamma^\star=\sum_{n=ad} \sum_{-d-\frac{a}{2}<k<d-\frac{a}{2}}
\left(\begin{smallmatrix} a+k & k\\d & d \end{smallmatrix}\right)^\star.
\end{equation}

Since $m\rightarrow m'S$ takes $\XX_n^<$ bijectively onto $\XX_n^>$, we have by \eqref{P2},
\eqref{P3} 
\begin{equation}\label{5.23}
\sum_{\gamma\in \XX_n^<} \gamma^\star=\sum_{\gamma\in \XX_n^>} \gamma^\star.
\end{equation}
It remains to calculate $\sum_{\gamma\in \XX_n^<} \gamma^\star$. 

For $m\in M_n$, denote by $\{m\}$ the equivalence class of $m$ in $\XX_n^<$ modulo
multiplication by powers of $T$ on the left, namely $\{m\}:=\Gamma_{1\infty} m \cap
\XX_n^<$ where $\Gamma_{1\infty}= \{T^k: k\in \Z\}$. Consider the involution 
\[
f:\Gamma_{1\infty}\backslash \XX_n^<\rightarrow \Gamma_{1\infty}\backslash \XX_n^<, \quad
f(\{m\})=\{m'\}.
\]
We show that $f$ is bijective, and the classes
$\{m\}$, $m\in\XX_n$ have one or two elements.
Let $m=\left(\begin{smallmatrix} a & -b\\c & d
\end{smallmatrix}\right)\in \XX_n^<$, and let $T^{-k} m'=  \left(\begin{smallmatrix} A & -B\\C
& D
\end{smallmatrix}\right)$. Then 
\[
T^{-k} m'\in\XX_n \Longleftrightarrow k>1+\frac{b-a}{c+d}, \quad
\Big|k-\frac{a+b}{d-c}\Big|<1. 
\]
Since $c+d>|b-a|$, $a+b>d-c>0$, it is clear that there exist one or two values of $k$ for
which $T^{-k} m'\in\XX_n$, and therefore $f$ is bijective (since it is an involution).
Moreover,
\begin{equation}
|\{m'\}|= \begin{cases}
           1 & \text{ if } d-c|a+b \quad (k=\frac{a+b}{d-c}); \\
				1 & \text{ if } d-c\nmid a+b,\ b \ge a,\ \frac{a+b}{d-c}<2  \quad (k=2);\\
			 2 & \text{ otherwise }.
          \end{cases}
\end{equation}
Since $d-c|a+b \Longleftrightarrow A=B$, and $a=b\Longleftrightarrow D-C|A+B$, and 
$b>a,\frac{a+b}{d-c}<2\Longleftrightarrow B>A, \frac{A+B}{D-C}<2$, there are four
possibilities:
\begin{enumerate}
 \item $d-c\nmid a+b,\ a\ne b \Longleftrightarrow D-C\nmid A+B, A\ne B$. Then
$|\{m\}|=|\{m'\}|$. 
 \item $d-c| a+b,\ a= b \Longleftrightarrow D-C|A+B, A= B$. Then  $|\{m\}|=|\{m'\}|=1$. 
 \item $d-c\nmid a+b,\ a= b \Longleftrightarrow D-C|A+B, A\ne B$. Then
$$|\{m\}|=1,\ |\{m'\}|=\begin{cases} 1 & \text{ if } \frac{a+b}{d-c}<2\\ 
                       2& \text{ if } \frac{a+b}{d-c}>2\end{cases}.
$$
  \item $d-c| a+b,\ a \ne b \Longleftrightarrow D-C\nmid A+B, A=B$. Then
$$|\{m'\}|=1,\ |\{m\}|=\begin{cases} 1 & \text{ if } \frac{A+B}{D-C}<2\\ 
                       2& \text{ if } \frac{A+B}{D-C}>2\end{cases}.
$$
\end{enumerate}
Summing over two copies of $\XX_n^<$ divided into classes $\{m\}$ and respectively $\{m'\}$,
everything cancels by \eqref{P2} except for matrices $m'$ in case (3) with
$\frac{a+b}{d-c}>2$, and matrices $m$ in case (4) for which $\frac{A+B}{D-C}>2$. The two sums
are equal and we obtain 
\[
2\sum_{\gamma\in \XX_n^<} \gamma^\star=
2 \sum_{\gamma} \gamma^\star, \quad\text{over } \gamma=\left(\begin{smallmatrix} a & a\\-c & d
\end{smallmatrix}\right), d>|c|, d-c\nmid 2a, \frac{2a}{d-c}>2.
\] 
Writing $n=a r$, so that $c+d=r$, we obtain (after substituting $-c$ for $c$ ) 
\[
\sum_{\gamma\in \XX_n^<} \gamma^\star\equiv \sum_{n=ar} \sum_{\substack{0<r+2c<a\\r+2c\nmid
2a}}\left(\begin{smallmatrix} a & a\\c & c+r
\end{smallmatrix}\right).
\] 
By \eqref{5.221}, \eqref{5.23},  
the proof of \eqref{5.21} is finished once we show that the same sum, but with
the congruence condition replaced by $r+2c|2a$, yields 0 (mod $\cJ+\cJ^\vee$). By \eqref{P2},
\eqref{P4}, it is enough to show that 
\[
\left(\begin{smallmatrix} a & a\\c & c+r
\end{smallmatrix}\right)= \left(\begin{smallmatrix} a'-tc' & -a'+t(c'+r')\\-c' & c'+r'
\end{smallmatrix}\right)  \Longleftrightarrow r+2c|2a, \ r'+2c'|2a', 
\]
where $n=ar=a'r'$, $t\in \Z$, and $0<r+2c<a$, $0<r'+2c'<a'$. Indeed the equality of the
matrices implies $t=\frac{2a}{r+2c}=\frac{2a'}{r'+2c'}$, proving the equivalence above.   

\end{proof}

We end this section with a proof of the second theta series identity in the introduction, which is entirely similar
to that of Corollary \ref{cor5.6}. Let $\SS_n',\cdots, \XX_n'$ be  the sets defined like $\SS_n,\cdots, 
\XX_n$ but with the extra parity conditions $a\equiv d \pmod{2}$, $b\equiv c \pmod{2}$. 

\begin{prop} \label{prop5.7}a) If $n$ is odd then 
\begin{equation}\label{5.25}
\sum_{\gamma\in\XX_n'}\gamma^\star\equiv\sum_{n=ad}
2\sum_{-\frac{a}{2} <k<\frac{d-a}{2}} \left(\begin{smallmatrix} d & d\\k & a+k
\end{smallmatrix}\right)^\star.
\end{equation}

b) If $n$ is even then 
\begin{equation}\label{5.26}
\sum_{\gamma\in\XX_n'}\gamma^\star\equiv\sum_{\substack{n=ad\\2|a}}\sum_{\substack{-d-
\frac{a}{2} <k<d-\frac{a}{2}\\k\equiv d\!\!\!\! \pmod{2}}}
\left(\begin{smallmatrix} a+k & k\\d & d \end{smallmatrix}\right)^\star+
\sum_{\substack{n=ad\\2|a,2|d}}2\sum_{-\frac{a}{2} <k<\frac{d-a}{2}} \left(\begin{smallmatrix}
d & d\\k & a+k
\end{smallmatrix}\right)^\star.
\end{equation}
\end{prop}
\begin{proof}
The proof is very similar to that of Proposition \ref{prop_alg}, and we prove only part a)
leaving b) as an exercise for the reader. Assuming $n$ odd and referring to the proof of
Proposition \ref{prop_alg}, the parity condition implies that $a\not\equiv b$, $c\not\equiv d$
(mod 2), so that $\XX_n^{\prime=}$ is empty, and only cases (1) and (4)
can occur for $m\in \XX_n^{\prime<}$. The integer $k$ such that $T^{-k} m'\in \XX_n^{\prime<}$
must be even, so in case (1) we have $|\{m\}|=|\{m'\}|=1$, while in case (4) we have
$|\{m'\}|=0$, because $\frac{a+b}{d-c}$ is odd in that case. Denoting by $\XX_n^{\prime(4)}$
the set of $m=\left(\begin{smallmatrix} a & -b\\c & d
\end{smallmatrix}\right)\in\XX_n'$, with $d>|c|, d-c|a+b$, it follows that 
\[
\sum_{\gamma\in\XX_n^{'<}} \gamma^\star\equiv\sum_{m\in\XX_n^{\prime(4)}} m^\star
\]
Writing $d-c=r>0, a+b=lr$, from $a(d-c)+c(a+b)=n$ we have that $n=rs$ with $s\in\Z$ and
$a+lc=s$, so $T^l m= \left(\begin{smallmatrix} s & s\\c & c+r
\end{smallmatrix}\right)$. The condition $m\in \XX_n^{\prime<}$ implies that 
\[ l>1, l\text{ odd, and } \left|l-\frac{2s}{2c+r}\right|<1, 
\]
so there is a unique such $m$ for each $c>\frac{-r}{2}$ such that $\frac{2s}{2c+r}>2$, and
$2c+r\nmid 2s$. Consequently:
\[
\sum_{\gamma\in\XX_n^{'<}}\gamma^\star\equiv\sum_{n=rs}\sum_{\substack{\frac{-r}{2}<c<\frac{s-r
} { 2 } \\2c+r\nmid 2s}}\left(\begin{smallmatrix} s & s\\c & c+r
\end{smallmatrix}\right)^\star .
\]
The proof is finished by observing that the divisibility condition can be removed, exactly as
in the previous case. 
\end{proof}

\begin{corollary}\label{cor5.8} We have the identity
\begin{equation}\label{5.27}
\sideset{}{'}\sum_{\substack{x\ge|y|, z\ge
|t|\\x>|t|,z>|y|}}\hspace{-8pt}q^{x^2+z^2-y^2-t^2}=\tE_2(q)-2\tE_2(q^2)+4\tE_2(q^4)-2\sum_{n>0}
\sigma_{\min}^\ev(n)\, q^n+\sum_{n>0} q^{n^2},
\end{equation}
where $\sigma_{\min}^\ev (n)=\sum_{n=ad, 2|(d-a)} \min(a,d)$.
\end{corollary}
\begin{proof}
Arguing as before, we have that $\NN(\XX_n')$ is given by \eqref{5.222} with $\XX_n$ replaced
by $\XX_n'$. Making a substitution $a=x+y,d=x-y, b=z+t,c=z-t$, the conditions $a,b,c,d\ge 0$
become $x\ge|y|$, $z\ge |t|$, while $a+b>|c-d|, c+d>|a-b|$ become $x>|t|, z> |y|$. Therefore
$\NN(\XX_n')$ is the coefficient of $q^n$ in the LHS of \eqref{5.27}.

As before, we can count $\NN(\XX_n')$ from Proposition \ref{prop5.7}. When $n$ is odd, counting
matrices in the RHS of \eqref{5.25}, with sign depending on the sign of $k$, gives 
\[\NN(\XX_n')=\sigma(n)-2\sigma_{\min}(n)+\delta(\sqrt{n}).
\] 
If $n$ is even, counting matrices in the RHS of \eqref{5.26} similarly yields (the first term comes
from the first sum, and the second from the second sum)
\[
\NN(\XX_n')=\big[
\sigma\big(\textstyle\frac{n}{2}\big)-\tau\big(\frac{n}{4}\big)\big]+
\big[2\sigma\big(\textstyle\frac{n}{4}\big)+\tau\big(\frac{n}{4}\big)-2\sigma_{\min}^\ev(n)+
\delta\big(\frac{\sqrt{n}}{2}\big) \big],
\] 
where $\tau(n)$ is the number of divisors of $n$ and we adopt the convention that an
arithmetic function is zero on non-integers. This is exactly the coefficient of $q^n$ in the
RHS of \eqref{5.27}
\end{proof}
The group ring relations in Propositions \ref{prop3}, \ref{prop_alg}, \ref{prop5.7}, as well as the
formulas in Corollaries \ref{cor5.6} and \ref{cor5.8}, have been checked numerically for  $n\le 100$
using MAGMA \cite{Mgm}.

\section*{Acknowledgments} Part of this work was completed at the Max Planck Institute in Bonn, which
provided financial support and a great working environment. We would like to thank Don Zagier for inspiring
conversations. The first author was partially supported by the CNCSIS grant PN-II-ID-PCE-2012-4-0376. 
The second author was partially supported by the European Community grant PIRG05-GA-2009-248569 and by the CNCS -
UEFISCDI grant PN-II-RU-TE-2011-3-0259.

\end{document}